\documentclass[a4paper,english]{article}
\usepackage[a4paper]{geometry}
\usepackage{amsmath,amsthm,amssymb,mathtools,fontenc}
\usepackage{enumerate,csquotes,verbatim,enumitem}
\usepackage{setspace}
\usepackage[numbers]{natbib}
\usepackage[T1]{fontenc}
\usepackage{hyperref}
\usepackage{color,xcolor}
\usepackage{tikz}
\usetikzlibrary{arrows}
\usepackage{multirow}
\usepackage{graphicx}
\usepackage{tabularx}
\usepackage{float} 		
\restylefloat{table}    

\newtheorem{thm}{Theorem}
\newtheorem{lem}[thm]{Lemma}
\newtheorem{cor}[thm]{Corollary}

\newtheorem{con}[thm]{Conjecture}  
\newtheorem{obs}[thm]{Observation}
\theoremstyle{definition}

\title{The graph grabbing game on  blow-ups of trees and cycles}
\author{Sopon Boriboon\thanks{\,Department of Mathematics and Computer Science, Faculty of Science, Chulalongkorn University, Bangkok 10330, Thailand; \texttt{soponboriboon@gmail.com}.}
  \and Teeradej Kittipassorn\thanks{\,Department of Mathematics and Computer Science, Faculty of Science, Chulalongkorn University, Bangkok 10330, Thailand; \texttt{teeradej.k@chula.ac.th}.}}

\begin{document}
	
		\maketitle
		
\begin{abstract}
	The graph grabbing game is played on a non-negatively weighted connected graph by Alice and Bob who alternately claim a non-cut vertex from the remaining graph, where Alice plays first, to maximize the weights on their respective claimed vertices at the end of the game when all vertices have been claimed. Seacrest and Seacrest conjectured that Alice can secure at least half of the weight of every weighted connected bipartite even graph. Later, Egawa, Enomoto and Matsumoto partially confirmed this conjecture by showing that Alice wins the game on a class of weighted connected bipartite even graphs called {$K_{m,n}$-trees}. 
	We extend the result on this class to include a number of graphs, e.g. even blow-ups of trees and cycles.
\end{abstract}
	
\section{Introduction}

	A vertex $v$ of a connected graph $G$ is a \emph{cut vertex} if $G-v$ is disconnected. A graph $G$ is \emph{even} (resp. \emph{odd}) if the number of vertices of $G$ is even (resp. odd). A \emph{weighted graph} $G$ is a graph $G$ with a weighted function $w:V(G)\rightarrow\mathbb{R}^+\cup\{0\}$.

	The \emph{graph grabbing game} is played on a non-negatively weighted connected graph by two players: Alice and Bob who alternately claim a non-cut vertex from the remaining graph and collect the weight on the vertex, where Alice plays first. The aim of each player is to maximize the weights on their respective claimed vertices at the end of the game when all vertices have been claimed. Alice \emph{wins} the game if she gains at least half of the total weight of the graph. 
	
	The first version of the graph grabbing game appeared in the first problem in Winkler's puzzle book (2003)~\cite{Winkler}, where he gave a winning strategy for Alice on every weighted even path and he observed that there is a weighted odd path on which Alice
	cannot win. In 2009, Rosenfeld~\cite{Rosenfeld:gold} proposed the game for trees and call it the \emph{gold grabbing game}. In 2011, Micek and Walczak~\cite{Micek:subdividedstar} generalized the game to general graphs and call it the graph grabbing game. They showed that Alice can secure at least a quarter of the weight of every weighted even tree and they conjectured that Alice can in fact secure at least half of the weight of every weighted even tree. Later in 2012, Seacrest and Seacrest~\cite{Seacrest:tree} solved this conjecture by considering a vertex-rooted version of the game and they posed the following conjecture. 

\begin{con}[\cite{Seacrest:tree}] \label{conj}
	Alice wins the game on every weighted connected bipartite even graph.
\end{con}

	In 2018, Egawa, Enomoto and Matsumoto~\cite{Egawa:Kmntree} gave a supporting evidence for this conjecture. They generalized the proof of Seacrest and Seacrest by considering a set-rooted version of the game to prove that Alice wins the game on every weighted even \emph{$K_{m,n}$-tree}, namely a bipartite graph obtained from a complete bipartite graph $K_{m,n}$ on $[m+n]$ and trees $T_1,\dots,T_{m+n}$ by identifying vertex $i$ of $K_{m,n}$ with exactly one vertex of $T_i$ for each $i\in[m+n]$. 

	For a graph $G$ with vertices $v_1,\dots,v_k$ and non-empty  sets $V_1,\dots,V_k$, a \emph{blow-up} B($G$) of $G$ is a graph obtained from $G$ by replacing $v_1,\dots,v_k$ with $V_1,\dots,V_k$, respectively where, for each $i,j\in[k]$,  vertices $x\in V_i$ and $y\in V_j$ are adjacent in B($G$) if and only if $v_i$ and $v_j$ are adjacent in $G$. For a graph $G$ on $[k]$ and trees $T_1,\dots,T_k$, a \emph{$G$-tree} is a graph obtained from $G$ by identifying vertex $i$ of $G$ with  exactly one vertex of $T_i$ for each $i\in[k]$. For a tree $T$, we note that a B$(T)$-tree and B($C_{2n}$) are connected bipartite graphs, and a B$(T)$-tree is a $K_{m,n}$-tree when $T$ is the path on two vertices.
 	\begin{figure}[H] \centering
 		\begin{tikzpicture}[baseline=0.5ex,scale=0.65]	
 		\draw[thick] (1.5,2) -- (2.5,2.5) -- (2,3.5) -- (1,3.5);
 		\draw[thick] (2.5,2.5) -- (3.5,2) -- (4.5,2) ;
 		\draw[thick] (2,3.5) -- (3,4);
 		\draw[fill=black] (1.5,2) circle (2.5pt) node[below=2pt] {$v_5$};
 		\draw[fill=black] (1,3.5) circle (2.5pt) node[above=2pt] {$v_1$};
 		\draw[fill=black] (2.5,2.5) circle (2.5pt) node[above right=-1pt] {$v_4$};
 		\draw[fill=black] (2,3.5) circle (2.5pt) node[above=2pt] {$v_2$};
 		\draw[fill=black] (3.5,2) circle (2.5pt) node[below=2pt] {$v_6$};
 		\draw[fill=black] (3,4) circle (2.5pt) node[above=2pt] {$v_3$};
 		\draw[fill=black] (4.5,2) circle (2.5pt) node[below=2pt] {$v_7$};
 		\draw (2.5,0) node {$T$};
 		\draw[thick] (8,2) -- (9.5 ,3);
 		\draw[thick] (9.5 ,3) -- (8,2.5);
 		\draw[thick] (8,2.5) -- (9.5,2.5);
 		\draw[thick] (9.5,2.5) -- (8,2);
 		\draw[thick] (7,4.25) -- (8.5,4.5);
 		\draw[thick] (8.5,4.5) -- (7,4.75) ;
 		\draw[thick] (9.5 ,3) -- (8.5,4.5);
 		\draw[thick] (8.5,4.5) -- (9.5,2.5);
 		\draw[thick] (9.5 ,3) -- (11,1.5);
 		\draw[thick] (11,1.5) -- (9.5,2.5);
 		\draw[thick] (9.5 ,3) -- (11,2);
 		\draw[thick] (11,2) -- (9.5,2.5);
 		\draw[thick] (9.5 ,3) -- (11,2.5);
 		\draw[thick] (11,2.5) -- (9.5,2.5);
 		\draw[thick] (12.5,1.75) -- (11,2.5);
 		\draw[thick] (11,2.5) -- (12.5,2.25);	
 		\draw[thick] (12.5,1.75) -- (11,2);
 		\draw[thick] (11,2) -- (12.5,2.25);	
 		\draw[thick] (12.5,1.75) -- (11,1.5);
 		\draw[thick] (11,1.5) -- (12.5,2.25);
 		\draw[thick] (10,5.5) -- (8.5,4.5);
 		\draw[thick] (8.5,4.5) -- (10,6);
 		\draw[thick] (10,5) -- (8.5,4.5);	
 		\draw[fill=black] (8,2) circle (2.5pt);
 		\draw[fill=black] (8,2.5) circle (2.5pt);
 		\draw[fill=black] (7,4.25) circle (2.5pt);
 		\draw[fill=black] (7,4.75) circle (2.5pt);
 		\draw[fill=black] (9.5 ,3) circle (2.5pt);
 		\draw[fill=black] (9.5,2.5) circle (2.5pt);
 		\draw[fill=black] (8.5,4.5) circle (2.5pt);
 		\draw[fill=black] (11,1.5) circle (2.5pt);
 		\draw[fill=black] (11,2) circle (2.5pt);
 		\draw[fill=black] (11,2.5) circle (2.5pt);
 		\draw[fill=black] (10,5) circle (2.5pt);
 		\draw[fill=black] (10,5.5) circle (2.5pt);
 		\draw[fill=black] (10,6) circle (2.5pt);
 		\draw[fill=black] (12.5,1.75) circle (2.5pt);
 		\draw[fill=black] (12.5,2.25) circle (2.5pt);
 		\draw[rotate=0] (8,2.25) ellipse (3mm and 8mm);
 		\draw[rotate=0] (7,4.5) ellipse (3mm and 8mm);
 		\draw[rotate=0] (8.5,4.5) ellipse (3mm and 3.5mm);
 		\draw[rotate=0] (12.5,2) ellipse (3mm and 8mm);
 		\draw[rotate=0] (9.5,2.75) ellipse (3mm and 8mm);
 		\draw[rotate=0] (11,2) ellipse (3mm and 10mm);
 		\draw[rotate=0] (10,5.5) ellipse (3mm and 10mm);
 		\draw[fill=black] (8,1) node {$V_5$};
 		\draw[fill=black] (9.5,1.5) node {$V_4$};
 		\draw[fill=black] (11,0.5) node {$V_6$};
 		\draw[fill=black] (12.5,0.75) node {$V_7$};
 		\draw[fill=black] (7,5.75) node {$V_1$};
 		\draw[fill=black] (8.5,5.25) node {$V_2$};
 		\draw[fill=black] (10.75,5.5) node {$V_3$};
 		\draw (9.5,0) node {B$(T)$};
 		\draw[thick] (17,2) -- (18.5 ,3);
 		\draw[thick] (18.5 ,3) -- (17,2.5);
 		\draw[thick] (17,2.5) -- (18.5,2.5);
 		\draw[thick] (18.5,2.5) -- (17,2);
 		\draw[thick] (16,4.25) -- (17.5,4.5);
 		\draw[thick] (17.5,4.5) -- (16,4.75) ;
 		\draw[thick] (18.5 ,3) -- (17.5,4.5);
 		\draw[thick] (17.5,4.5) -- (18.5,2.5);
 		\draw[thick] (18.5,3) -- (20,1.5);
 		\draw[thick] (20,1.5) -- (18.5,2.5);
 		\draw[thick] (18.5,3) -- (20,2);
 		\draw[thick] (20,2) -- (18.5,2.5);
 		\draw[thick] (18.5 ,3) -- (20,2.5);
 		\draw[thick] (20,2.5) -- (18.5,2.5);
 		\draw[thick] (21.5,1.75) -- (20,2.5);
 		\draw[thick] (20,2.5) -- (21.5,2.25);	
 		\draw[thick] (21.5,1.75) -- (20,2);
 		\draw[thick] (20,2) -- (21.5,2.25);	
 		\draw[thick] (21.5,1.75) -- (20,1.5);
 		\draw[thick] (20,1.5) -- (21.5,2.25);
 		\draw[thick] (19,5.5) -- (17.5,4.5);
 		\draw[thick] (17.5,4.5) -- (19,6);
 		\draw[thick] (19,5) -- (17.5,4.5);		
 		\draw[fill=black] (17,2) circle (2.5pt);
 		\draw[fill=black] (17,2.5) circle (2.5pt);
 		\draw[fill=black] (16,4.25) circle (2.5pt);
 		\draw[fill=black] (16,4.75) circle (2.5pt);
 		\draw[fill=black] (18.5,3) circle (2.5pt);
 		\draw[fill=black] (18.5,2.5) circle (2.5pt);
 		\draw[fill=black] (17.5,4.5) circle (2.5pt);
 		\draw[fill=black] (20,1.5) circle (2.5pt);
 		\draw[fill=black] (20,2) circle (2.5pt);
 		\draw[fill=black] (20,2.5) circle (2.5pt);
 		\draw[fill=black] (19,5) circle (2.5pt);
 		\draw[fill=black] (19,5.5) circle (2.5pt);
 		\draw[fill=black] (19,6) circle (2.5pt);
 		\draw[fill=black] (21.5,1.75) circle (2.5pt);
 		\draw[fill=black] (21.5,2.25) circle (2.5pt);
 		\draw[rotate=0] (17,2.25) ellipse (2.5mm and 6mm);
 		\draw[rotate=0] (16,4.5) ellipse (2.5mm and 6mm);
 		\draw[rotate=0] (17.5,4.5) ellipse (2.5mm and 3.5mm);
 		\draw[rotate=0] (21.5,2) ellipse (2.5mm and 6mm);
 		\draw[rotate=0] (18.5,2.75) ellipse (2.5mm and 6mm);
 		\draw[rotate=0] (20,2) ellipse (2.5mm and 8mm);
 		\draw[rotate=0] (19,5.5) ellipse (2.5mm and 8mm);
 		\draw (19,0) node {B$(T)$-tree};
 		\draw[fill=black] (17.5,5) circle (2.5pt);
 		\draw[fill=black] (17,5.5) circle (2.5pt);
 		\draw[fill=black] (17.5,5.5) circle (2.5pt);
 		\draw[fill=black] (18,5.5) circle (2.5pt);
 		\draw[fill=black] (19.5,5.5) circle (2.5pt);
 		\draw[fill=black] (19.5,5) circle (2.5pt);
 		\draw[fill=black] (20,5) circle (2.5pt);
 		\draw[fill=black] (19.5,4.5) circle (2.5pt);
 		\draw[fill=black] (20,4.5) circle (2.5pt);
 		\draw[fill=black] (15.5,4) circle (2.5pt);
 		\draw[fill=black] (15.5,4.5) circle (2.5pt);
 		\draw[fill=black] (15.5,5) circle (2.5pt);
 		\draw[fill=black] (15,4) circle (2.5pt);
 		\draw[fill=black] (16.75,1.5) circle (2.5pt);
 		\draw[fill=black] (17.25,1.5) circle (2.5pt);
 		\draw[fill=black] (17.25,1) circle (2.5pt);
 		\draw[fill=black] (18.5,2) circle (2.5pt);
 		\draw[fill=black] (18.25,1) circle (2.5pt);
 		\draw[fill=black] (18.25,1.5) circle (2.5pt);
 		\draw[fill=black] (18.75,1.5) circle (2.5pt);
 		\draw[fill=black] (19.5,3) circle (2.5pt);
 		\draw[fill=black] (20,3) circle (2.5pt);
 		\draw[fill=black] (20.5,3) circle (2.5pt);
 		\draw[fill=black] (20.5,1) circle (2.5pt);
 		\draw[fill=black] (21,1) circle (2.5pt);
 		\draw[fill=black] (22,1.75) circle (2.5pt);
 		\draw[fill=black] (22,2.25) circle (2.5pt);
 		\draw[fill=black] (22.5,2.5) circle (2.5pt);
 		\draw[fill=black] (22.5,2) circle (2.5pt);
 		\draw[thick] (17.5,4.5) -- (17.5,5) -- (17.5,5.5);
 		\draw[thick] (17,5.5) -- (17.5,5) -- (18,5.5);
 		\draw[thick] (19,5.5) -- (19.5,5.5);
 		\draw[thick] (19,5) -- (19.5,5) -- (20,5); 
 		\draw[thick] (19,5) -- (19.5,4.5) -- (20,4.5); 
 		\draw[thick] (15.5,4.5) -- (16,4.75) -- (15.5,5);	
 		\draw[thick] (16,4.25) -- (15.5,4) -- (15,4);
 		\draw[thick] (16.75,1.5) -- (17,2) -- (17.25,1.5) -- (17.25,1);
 		\draw[thick] (18.25,1) -- (18.25,1.5) -- (18.5,2) -- (18.75,1.5);
 		\draw[thick] (18.5,2) -- (18.5,2.5);
 		\draw[thick] (19.5,3) -- (20,2.5) -- (20,3);
 		\draw[thick] (20,2.5) -- (20.5,3);
 		\draw[thick] (20,1.5) -- (20.5,1) -- (21,1);
 		\draw[thick] (21.5,1.75) -- (22,1.75);
 		\draw[thick] (22.5,2.5) -- (22,2.25) -- (22.5,2);
 		\draw[thick] (21.5,2.25) -- (22,2.25);
 		\end{tikzpicture} 
 		\caption{Examples of a tree $T$, a blow-up B$(T)$ and a B$(T)$-tree.} \label{figure:Btree new}
 	\end{figure}
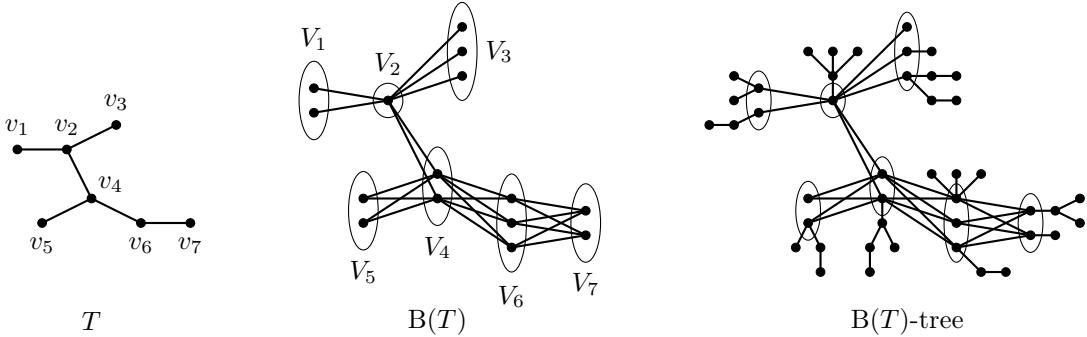

In this paper, we partially confirm Conjecture~\ref{conj} as follows.

\begin{thm} \label{thm:mainthm}
	Alice wins the game on every weighted even B$(T)$-tree, where $T$ is a tree.
\end{thm}

\begin{cor} \label{cor:blowupcycle}
	Alice wins the game on every weighted even B$(C_{n})$.
\end{cor}

	For a graph $G$ and a set $S\subseteq V(G)$, let $N_G(S)$ denote the neighborhood of $S$, i.e. the set of vertices having a neighbor in $S$. The proof is based on the method of Egawa, Enomoto and Matsumoto, where their main lemmas dealt with the score of the game on a $K_{m,n}$-tree rooted at a partite class. We generalize their method by considering instead the scores of the game on a $H$-tree rooted at $V_i$ and  the game on $H$-tree rooted at $N_H(V_i)$, where $H$ is a blow-up of a tree. 

	The rest of this paper is organized as follows. In Section~\ref{scetion-Observations}, we recall some observations and a lemma on $K_{m,n}$-trees given by Egawa, Enomoto and Matsumoto. Section~\ref{scetion-The Proofs} is devoted to proving Theorem~\ref{thm:mainthm} and then applying it to prove Corollary~\ref{cor:blowupcycle}. In Section~\ref{section-conclusion}, we give some concluding remarks.

\section{Preliminaries} \label{scetion-Observations}
	
	In this section, we prepare some observations and a lemma on $K_{m,n}$-trees which will be useful for the proof of Theorem~\ref{thm:mainthm}.

	We first give definitions of a rooted version of the graph grabbing game and some related terms introduced by Egawa, Enomoto and  Matsumoto. For a weighted graph $G$, a \emph{root set} $S$ of $G$ is a set of vertices intersecting every component of $G$ and the \emph{game on $G$ rooted at $S$} is a graph grabbing game, where each player needs not claim a non-cut vertex, but instead they claim a vertex $v$ such that every component of $G-v$ contains at least one vertex in $S$. Therefore, a move $v$ in the game on $G$ is \emph{feasible} if $G-v$ is connected, and a move $v$ in the game on $G$ rooted at $S$ is \emph{feasible} if every component of $G-v$ contains at least one vertex in $S$. A move $v$ in the game on $G$ (rooted at $S$) is \emph{optimal} if there is an optimal strategy in the game on $G$ (rooted at $S$) having $v$ as the first move. The first (resp. second) player is called \emph{Player $1$} (resp. \emph{Player $2$}). The last (resp. second from last) player is called \emph{Player $-1$} (resp. \emph{Player $-2$}). For $k\in\{1,2,-1,-2\}$, assuming that both players play optimally, let $N(G,k)$ denote the score of Player $k$ in the game on $G$ and let $R(G,S,k)$ denote the score of Player $k$ in the game on $G$ rooted at $S$ and we write $R(G,v,k)$ for $R(G,\{v\},k)$. For a set $S$ and an element $x$, we write $S-x$ for $S\setminus\{x\}$.
	
	Egawa, Enomoto and Matsumoto observed some relationships between the scores of both players in the normal version and the rooted version of the game. Note that the equation/inequality in the brackets in each observation is an equivalent form of the first one due to the fact that, assuming that both players play optimally, the sum of their scores equals to the total weight of the graph.
	
\begin{obs}[\cite{Egawa:Kmntree}] \label{obs1} 
	If $x$ is a feasible move in the game on $G$, then
	\begin{center}
		$N(G,2)\leq N(G-x,1)$ \quad $(\Leftrightarrow N(G,1)\geq N(G-x,2)+w(x))$.
	\end{center}
	If $x$ is an optimal move in the game on $G$, then 	
	\begin{center}
		$N(G,2)= N(G-x,1)$ \quad $(\Leftrightarrow N(G,1) = N(G-x,2)+w(x))$.
	\end{center}
\end{obs}

\begin{obs}[\cite{Egawa:Kmntree}] \label{obs2}
	Let $S$ be a root set of $G$. If $x$ is a feasible move in the game on $G$ rooted at $S$, then
	\begin{center}
		$R(G,S,2)\leq R(G-x,S-x,1)$ \quad $(\Leftrightarrow R(G,S,1)\geq R(G-x,S-x,2)+w(x))$.
	\end{center}
	If $x$ is an optimal move in the game on $G$ rooted at $S$, then 
	\begin{center}
		$R(G,S,2)= R(G-x,S-x,1)$ \quad $(\Leftrightarrow R(G,S,1)= R(G-x,S-x,2)+w(x))$.
	\end{center}
\end{obs}

\begin{obs}[\cite{Egawa:Kmntree}] \label{obs3}
	If $v$ is a root of $G$, then
	\begin{center}
		$R(G,v,-2)=R(G-v,N_G(v),-1)$ \quad $(\Leftrightarrow R(G,v,-1)=R(G-v,N_G(v),-2)+w(v))$.
	\end{center}	
\end{obs}

The next lemma is a part of their main results which will help us in the proof.

\begin{lem}[\cite{Egawa:Kmntree}] \label{lemKmn}
	Let $G$ be a $K_{m,n}$-tree with partite classes $X,Y$ of size $m, n\geq1$, respectively. Then
	\begin{center}
		$R(G,Y,-2)\leq N(G,-2)$ \quad $(\Leftrightarrow R(G,Y,-1)\geq N(G,-1))$.
	\end{center}
\end{lem}

\section{The Proofs} \label{scetion-The Proofs}
	
	In this section, we start by proving Lemma~\ref{lem:gameB} which will be used repeatedly in the proof of our main lemmas, namely, Lemmas~\ref{lem1} and~\ref{lem2}. We then prove Theorem~\ref{thm:mainthm} by applying the main lemmas and deduce Corollary~\ref{cor:blowupcycle} from Theorem~\ref{thm:mainthm}.
	
	The following lemma shows the relationship between the scores of both players in the game on an even graph rooted at two different sets of some structure.
	
\begin{lem} \label{lem:gameB}
	Let $G_1$ and $G_2$ be subgraphs of an even graph $G$ such that $V(G_1)$, $V(G_2)$ partition $V(G)$. If $U_1=V(G_1)\cap N_G(V(G_2))$ and $U_2=V(G_2)\cap N_G(V(G_1))$ are root sets of $G_1$ and $G_2$, respectively, and every vertex in $U_1$ is joined to every vertex in $U_2$, then
\begin{enumerate}  [label={\thelem.\arabic*}]
	\item \label{lemB.1} 	$R(G,U_1,1)\geq R(G_1,U_1,-2)+R(G_2,U_2,-1)$.
	\item \label{lemB.2} $R(G,U_1,1)\geq R(G,U_2,2)$ 
\end{enumerate}	
\end{lem}

\begin{proof} 
	\begin{figure}[H] \centering
		\begin{tikzpicture}[baseline=0.5ex,scale=0.75]	
		\draw[thick,line width=0.5pt] (6,0.7) -- (8,0.7);
		\draw[thick,line width=0.5pt] (6,3.3) -- (8,3.3);
		\draw[rounded corners=5pt] (7.5,0.5) rectangle (11,3.5);		
		\draw[rounded corners=5pt] (3,0.5) rectangle (6.5,3.5);
		\draw[fill=white] (6,2) ellipse (4mm and 13mm);
		\draw[fill=white] (8,2) ellipse (4mm and 13mm);
		\draw (5.25,2) node {$U_1$};
		\draw (8.75,2) node {$์U_2$};
		\draw (2.5,2) node {$G_1$};
		\draw (11.5,2) node {$G_2$};
		\end{tikzpicture} 
		\caption{The graph $G$ in Lemma~\ref{lem:gameB}.} \label{fig:gameB}
	\end{figure}
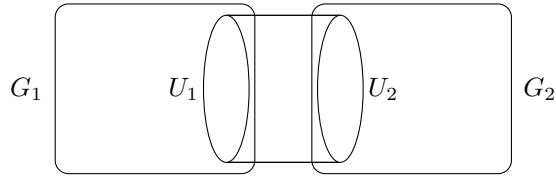	
	First, we shall prove Lemma~\ref{lemB.1} 
	by considering a strategy for Alice who plays first in the game on $G$ rooted at $U_1$. She plays optimally as Player~$-2$ in the game on $G_1$ rooted at $U_1$ and plays optimally as Player $-1$ in the game on $G_2$ rooted at $U_2$. Since $|V(G_1)|+|V(G_2)|$ is even, she plays as Player $1$ in one game and as Player $2$ in the other. Now, we check that Alice's moves are feasible in the game on $G$ rooted at $U_1$, and Bob's moves are feasible in the game on $G_1$ rooted at $U_1$ and the game on $G_2$ rooted at $U_2$. Indeed, after each move of Alice, every remaining component of $G_1$ and $G_2$ contains a vertex in $U_1$ and $U_2$, respectively. Together with the fact that every vertex in $U_2$ is joined to the remaining subset of $U_1$, we can conclude that every remaining component of $G$ contains a vertex in $U_1$. That is, her moves are feasible in the game on $G$ rooted at $U_1$. On the other hand, after each move of Bob, every remaining component of $G$ contains a vertex of $U_1$. Since the edges between $G_1$ and $G_2$ have endpoints only in $U_1$ and $U_2$, every remaining component of $G_1$ or $G_2$ contains a vertex in $U_1$ or $U_2$, respectively. That is, his moves are feasible in the game on $G_1$ rooted at $U_1$ and the game on $G_2$ rooted at $U_2$. Hence
	\begin{equation*} \label{eq:gameB1}
	R(G,U_1,1)\geq R(G_1,U_1,-2)+R(G_2,U_2,-1),
	\end{equation*}
	which completes the proof of Lemma~\ref{lemB.1}. 	
	By symmetry, we have
	\begin{equation*} \label{eq:gameB*}
	R(G,U_2,1)\geq R(G_1,U_1,-1)+R(G_2,U_2,-2),
	\end{equation*}
	which is equivalent to
	\begin{equation*}  \label{eq:gameB2}
	R(G,U_2,2)\leq R(G_1,U_1,-2)+R(G_2,U_2,-1),
	\end{equation*}
	by considering the total weight of $G, G_1$ and $G_2$. Together with Lemma~\ref{lemB.1}, we have	 
	\allowdisplaybreaks
	\begin{equation*}
		R(G,U_2,2)\leq R(G_1,U_1,-2)+R(G_2,U_2,-1)\leq R(G,U_1,1), 
	\end{equation*}
	which completes the proof of Lemma~\ref{lemB.2}.	 
\end{proof}

	We are now ready to prove the main lemmas which generalize the results on $K_{m,n}$-trees to B$(T)$-trees relating the scores of both players in the normal version and the rooted version of the game.
	 
\begin{lem} \label{lem1} Let $H$ be a blow-up graph of a tree with sets of vertices $V_1,\dots,V_k$ and let $G$ be a $H$-tree.
	\begin{enumerate}  [label={\thelem.\arabic*}]
		\item \label{lem1.1} For a vertex $v\in V(G)$, $R(G,v,-2)\leq N(G,-2)$ \quad $(\Leftrightarrow R(G,v,-1)\geq N(G,-1))$.
		\item \label{lem1.2} For each $i\in[k]$, $R(G,V_i,-2)\leq N(G,-2)$ \quad $(\Leftrightarrow R(G,V_i,-1)\geq N(G,-1))$.
		\item \label{lem1.3} For each $i\in[k]$, $R(G,N_H(V_i),-2)\leq N(G,-2)$ \quad $(\Leftrightarrow R(G,N_H(V_i),-1)\geq N(G,-1))$.
	\end{enumerate}
\end{lem}

\begin{lem} \label{lem2} Let $H$ be a blow-up graph of a tree with sets of vertices $V_1,\dots,V_k$ and let $G$ be an even $H$-tree.
	\begin{enumerate} 	[label={\thelem.\arabic*}]
		\item \label{lem2.1} For a vertex $v\in V(G),$ $R(G,v,1)\geq N(G,2)$ \quad $(\Leftrightarrow R(G,v,2)\leq N(G,1))$.
		\item \label{lem2.2} For each $i\in[k]$, $R(G,V_i,1)\geq N(G,2)$ \quad $(\Leftrightarrow R(G,V_i,2)\leq N(G,1))$.
		\item \label{lem2.3} For each $i\in[k]$, $R(G,N_H(V_i),1)\geq N(G,2)$ \quad $(\Leftrightarrow R(G,N_H(V_i),2)\leq N(G,1))$.
	\end{enumerate}
\end{lem}

	We prove Lemmas~\ref{lem1} and~\ref{lem2} simultaneously by induction on the number $n$ of vertices of $G$. It is easy to check that Lemmas~\ref{lem1} and~\ref{lem2} hold for $n\leq2$. Now, we let $n\geq3$ and suppose that Lemmas~\ref{lem1} and~\ref{lem2} hold for $|V(G)|<n$. We remark that the following fact will be used throughout the proofs: Let $G$ be a $H$-tree, where $H$ is a blow-up of a tree and let $v$ be a vertex in $G$. Then $G-v$ is a $H'$-tree, where $H'$ is a blow-up of some tree if and only if $G-v$ is connected.

\begin{proof}[Proof of Lemma~\ref{lem1.1}]  Let $v\in V(G)$. 
	
	\textbf{Case 1.} $G$ is even. 
	
	Let $a$ be an optimal move in the game on $G$ rooted at $v$. Therefore, $a\neq v$ and $a$ is feasible in the game on $G$. So $G-a$ is connected. Then \allowdisplaybreaks
	\begin{align*}
	R(G,v,-1=2)&=R(G-a,v,1=-1)     &(\text{Observation~\ref{obs2}})   \\
	&\geq N(G-a,-1=1)      &(\text{Lemma~\ref{lem1.1} by induction})   \\
	&\geq N(G,2=-1)        &(\text{Observation~\ref{obs1}}).
	\end{align*}
	
	\textbf{Case 2.} $G$ is odd. 
	
	Let $b$ be an optimal move in the game on $G$. So $G-b$ is connected.
	
	\textbf{Case 2.1.}  $b\neq v$. 
	
	Therefore, $b$ is a feasible move in the game on $G$ rooted at $v$. Then \allowdisplaybreaks
	\begin{align*}
	R(G,v,-2=2)&\leq R(G-b,v,1=-2)    &(\text{Observation~\ref{obs2}}) \\
	&\leq N(G-b,-2=1)   &(\text{Lemma~\ref{lem1.1} by induction}) \\
	&= N(G,2=-2)        &(\text{Observation~\ref{obs1}}).
	\end{align*}
	
	\textbf{Case 2.2.}  $b=v$ and $v$ is a leaf. 
	
	Let $u$ be the unique neighbor of $v$. Then  \allowdisplaybreaks
	\begin{align*}
	R(G,v,-2)&= R(G-v,u,-1=2)     &(\text{Observation~\ref{obs3}})    \\
	&\leq N(G-v,1)     &(\text{Lemma~\ref{lem2.1} by induction})   \\
	&= N(G,2=-2)       &(\text{Observation~\ref{obs1} and }b=v).  
	\end{align*}
		
	\textbf{Case 2.3.}  $b=v$ and $v$ is not a leaf. 
	
	Therefore, $v\in V_i$ for some $i\in[k]$ and  $N_G(v)=N_H(V_i)$. Then \allowdisplaybreaks
	\begin{align*}
	 R(G,v,-2)&= R(G-v,N_G(v)=N_H(V_i),-1=2) &(\text{Observation~\ref{obs3}})   \\ 		   
	&\leq N(G-v,1)     & (\text{Lemma~\ref{lem2.3} by induction})   \\
	&= N(G,2=-2)       & (\text{Observation~\ref{obs1} and }b=v).     
	& \qedhere 
	\end{align*} 	
\end{proof} 

\begin{proof}[Proof of Lemma~\ref{lem1.2}] Let $i\in[k]$. If $|V_i|=1$, then we are done by Lemma~\ref{lem1.1}. Now, suppose that $|V_i|\geq 2$. 
		
	\textbf{Case 1.} $G$ is odd. 
	
	Let $b$ be an optimal move in the game on $G$. So $G-b$ is connected. Since $|V_i|\geq 2$, we have $V_i-b\neq\emptyset$. Therefore, $b$ is a feasible move in the game on $G$ rooted at $V_i$. Then  \allowdisplaybreaks
	\begin{align*}
	N(G,-2=2)&= N(G-b,1=-2)  &(\text{Observation~\ref{obs1}})     \\
	&\geq R(G-b,V_i-b,-2=1)  &(\text{Lemma~\ref{lem1.2} by induction})    \\
	&\geq R(G,V_i,2=-2)   
	&(\text{Observation~\ref{obs2}}).                            
	\end{align*}	
		
	\textbf{Case 2.} $G$ is even. 
	
	Let $a$ be an optimal move in the game on $G$ rooted at $V_i$. 
	
	\textbf{Case 2.1.} $a$ is a feasible move in the game on $G$. 
	
	So $G-a$ is connected. Then \allowdisplaybreaks
	\begin{align*}
	R(G,V_i,-1=2)&= R(G-a,V_i-a,1=-1)   & (\text{Observation~\ref{obs2}})       \\
	&\geq N(G-a,-1=1)    & (\text{Lemma~\ref{lem1.2} by induction}) 	   \\
	&\geq N(G,2=-1)      & (\text{Observation~\ref{obs1}}).                   
	\end{align*}	 
	
	\textbf{Case 2.2.} $a$ is not a feasible move in the game on $G$. 
	\begin{figure}[H] \centering
		\begin{tikzpicture}[baseline=0.5ex,scale=0.75]	
		\draw[thick,line width=0.5pt] (6,1.25) -- (7.5,2);
		\draw[thick,line width=0.5pt] (6,2) -- (7.5,2);
		\draw[thick,line width=0.5pt] (6,2.75) -- (7.5,2);
		\draw (6,2) ellipse (4mm and 13mm);
		\draw[fill=black] (6,1.25) circle (2pt);
		\draw[fill=black] (6,2) circle (2pt);	
		\draw[fill=black] (6,2.75) circle (2pt);
		\draw (5,1.25) ellipse (12mm and 3mm);
		\draw (5,2) ellipse (12mm and 3mm);
		\draw (5,2.75) ellipse (12mm and 3mm);
		\draw[fill=black] (7.5,2) circle (2pt) node[below=2pt] {$a$};
		\draw (7.5,2) ellipse (2.5mm and 6mm);
		\draw (7.5,1) node {$V_j$};
		\draw (6,0.25) node {$V_i$};
		\end{tikzpicture} 
		\caption{The graph $G$ in Case 2.2 of Lemma~\ref{lem1.2}.} \label{fig:1.2}
	\end{figure}
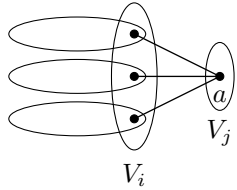		
	So $G-a$ is disconnected. Since $a$ is a feasible move  in the game on $G$ rooted at $V_i$, we have $a\in V_j$ for some $j\in[k]$ and $N_G(V_j)=N_H(V_j)$. Since $G-a$ is disconnected, $V_j=\{a\}$ and $a$ is not a leaf. If $i=j$, then every component of $G-a$ does not contain  a vertex in $V_i$. If there is a vertex set $V_l$, where $l\notin\{i,j\}$, then either $G-a$ is connected or there is a component of $G-a$ which does not contain a vertex in $V_i$. Hence $V_j=\{a\}$ for some $j\neq i$, $N_H(V_j)=V_i$ and $N_H(V_i)=V_j$.
	Therefore, $G$ is a $K_{m,n}$-tree with partite classes $V_i$, $V_j$. Then, by Lemma~\ref{lemKmn},
	\begin{equation*}
		N(G,-1)\leq R(G,V_i,-1). \qedhere
	\end{equation*}
\end{proof}

\begin{proof}[Proof of Lemma~\ref{lem1.3}]
	Let $i\in[k]$. If $|N_H(V_i)|=1$ or $N_H(V_i)=V_j$ for some $j\in[k]$, then we are done by Lemmas~\ref{lem1.1} or~\ref{lem1.2}, respectively. Now, suppose that $|N_H(V_i)|\geq 2$ and $V_i$ is joined to at least two sets in $V_1,\dots,V_k$.
	
	\textbf{Case 1.}  $G$ is odd. 
	
	Let $b$ be an optimal move in the game on $G$. So $G-b$ is connected. Since $|N_H(V_i)|\geq 2$, we have $N_H(V_i)-b\neq\emptyset$.  Then $b$ is a feasible move in the game on $G$ rooted at $N_H(V_i)$. Then \allowdisplaybreaks
	\begin{align*}
	N(G,-2=2)&=N(G-b,1=-2)   	& (\text{Observation~\ref{obs1}})   \\
	&\geq R(G-b,N_H(V_i)-b,-2=1)  &(\text{Lemma~\ref{lem1.3} by induction})      \\
	&\geq R(G,N_H(V_i),2=-2)      &(\text{Observation~\ref{obs2}}).                        
	\end{align*}	
		
	\textbf{Case 2.} $G$ is even. 
	
	Let $a$ be an optimal move in the game on $G$ rooted at $N_H(V_i)$. 
	
	\textbf{Case 2.1.} $a$ is a feasible move in the game on $G$. 
	
	So $G-a$ is connected. Then \allowdisplaybreaks
	\begin{align*}
	R(G,N_H(V_i),-1=2)&= R(G-a,N_H(V_i)-a,1=-1) & (\text{Observation~\ref{obs2}})  \\
	&\geq N(G-a,-1=1)     	& (\text{Lemma~\ref{lem1.3} by induction})     \\
	&\geq N(G,2=-1)      	& (\text{Observation~\ref{obs1}}).               
	\end{align*}
	
	\textbf{Case 2.2.} $a$ is not a feasible move in the game on $G$. 

	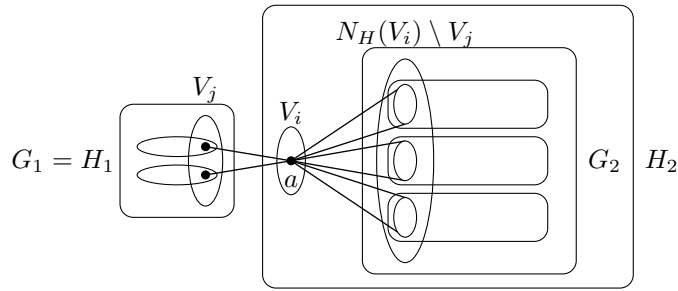
\begin{figure}[H] \centering
		\begin{tikzpicture}[baseline=0.5ex,scale=0.75]	
		\draw[thick,line width=0.5pt] (7.5,2) -- (9.5,3.35);
		\draw[thick,line width=0.5pt] (7.5,2) -- (9.5,2.65);
		\draw[thick,line width=0.5pt] (7.5,2) -- (9.5,2.35);
		\draw[thick,line width=0.5pt] (7.5,2) -- (9.5,1.65);
		\draw[thick,line width=0.5pt] (7.5,2) -- (9.5,1.35);
		\draw[thick,line width=0.5pt] (7.5,2) -- (9.5,0.65);
		\draw[thick,line width=0.5pt] (6,1.75) -- (7.5,2);
		\draw[thick,line width=0.5pt] (6,2.25) -- (7.5,2);
		\draw[rounded corners=5pt] (7,-0.25) rectangle (13.5,4.75);	
		\draw[rounded corners=5pt] (8.75,0) rectangle (12.5,4);		
		\draw[rounded corners=5pt] (4.5,1) rectangle (6.5,3);
		\draw[rounded corners=5pt] (9.2,2.575) rectangle (12,3.425);
		\draw[rounded corners=5pt] (9.2,1.575) rectangle (12,2.425);
		\draw[rounded corners=5pt] (9.2,0.575) rectangle (12,1.425);	
		\draw (6,2) ellipse (3mm and 8mm);
		\draw (9.5,2) ellipse (5mm and 18mm);
		\draw[fill=white] (9.5,1) ellipse (2mm and 3.5mm);	
		\draw[fill=white] (9.5,2) ellipse (2mm and 3.5mm);
		\draw[fill=white] (9.5,3) ellipse (2mm and 3.5mm);		
		\draw (6,3.25) node {$V_j$};
		\draw (9.5,4.25) node {$์N_H(V_i)\setminus V_j$};
		\draw (3.5,2) node {$G_1=H_1$};
		\draw (13,2) node {$G_2$};
		\draw (14,2) node {$H_2$};
		\draw[fill=black] (6,1.75) circle (2pt);	
		\draw[fill=black] (6,2.25) circle (2pt);
		\draw (5.5,1.75) ellipse (7mm and 1.75mm);
		\draw (5.5,2.25) ellipse (7mm and 1.75mm);
		\draw[fill=black] (7.5,2) circle (2pt) node[below=2pt] {$a$};
		\draw (7.5,2) ellipse (2.5mm and 6mm);
		\draw (7.5,2.85) node {$V_i$};
		\end{tikzpicture} 
		\caption{The graph $G$ in Case 2.2 of Lemma~\ref{lem1.3}.} \label{fig:1.3}
	\end{figure}	
	So $G-a$ is disconnected. Since $a$ is a feasible move in the game on $G$ rooted at $N_H(V_i)$, we have $a\in V_j$ for some $j\in[k]$ and $N_G(V_j)=N_H(V_j)$. Since $G-a$ is disconnected, $V_j=\{a\}$ and $a$ is not a leaf. Suppose that $i\neq j$. Since $V_i$ is joined to at least two sets, $V_i$ and $N_H(V_i)$ lie in the same component of $G-a$, but then the other components of $G-a$ does not contain  a vertex in $N_H(V_i)$, which is a contradiction. Hence $V_i=\{a\}$. Let $V_j\subseteq N_H(V_i)$ and let $G_1$ be the union of components in $G-a$ containing some vertices of $V_j$ and let $G_2=G-a-G_1$. By assumption, $G_2$ is not empty. 
	
	First, we shall show that
	\begin{center}
		$R(G,N_H(V_i),-1)\geq R(G_1,V_j,-1)+ R(G_2,N_H(V_i)\setminus V_j,-1)$,
	\end{center}
	by considering a strategy for Bob who plays second in the game on $G$ rooted at $N_H(V_i)$ after Alice grabs~$a$. He plays optimally as Player $-1$ in the game on $G_1$ rooted at $V_j$ and plays optimally as Player $-1$ in the game on $G_2$ rooted at $N_H(V_i)\setminus V_j$. Since $|V(G_1)|+|V(G_2)|$ is odd, he plays as Player $1$ in one game and as Player $2$ in the other. Now, we check that Bob's moves are feasible in the game on $G$ rooted at $N_H(V_i)$ and Alice's moves are feasible in the game on $G_1$ rooted at $V_j$ and the game on $G_2$ rooted at $N_H(V_i)\setminus V_j$. Indeed, after each move of Bob, every remaining component in $G_1$ or $G_2$ contains a vertex in $V_j$ or $N_H(V_i)\setminus V_j$, respectively. Then every remaining component of $G$ contains a vertex in $N_H(V_i)$. That is, his moves are feasible in the game on $G$ rooted at $N_H(V_i)$. On the other hand, after each move of Alice, every remaining component of $G$ contains a vertex in $N_H(V_i)$. Then every remaining component of $G_1$ or $G_2$ contains a vertex in $V_j$ or $N_H(V_i)\setminus V_j$, respectively. That is, her moves are feasible in the game on $G_1$ rooted at $V_j$ and the game on $G_2$ rooted at $N_H(V_i)\setminus V_j$. Hence
	\begin{equation} \label{eq:gameA1}
	R(G,N_H(V_i),-1)\geq R(G_1,V_j,-1)+ R(G_2,N_H(V_i)\setminus V_j,-1).
	\end{equation}
	Next, we let $H_1=G_1$ and $H_2=G-G_1$. We observe that $V_j=V(H_1)\cap N_G(V(H_2))$ and $\{a\}=V(H_2)\cap N_G(V(H_1))$ are root sets of $H_1$ and $H_2$, respectively, and $a$ is adjacent to all vertices in $V_j$. 
	Hence 
	\begin{align*} 
	R(G,V_j,-2=1)&\geq R(G_1,V_j,-2)+R(G-G_1,a,-1)    &(\text{Lemma~\ref{lemB.1}}) \\
	&=R(G_1,V_j,-2)+R(G_2,N_H(V_i)\setminus V_j,-2)+w(a) 
	&(\text{Observation~\ref{obs3}}), 				            
	\end{align*}  	            
	which is equivalent to 
	\begin{equation}   \label{eq:gameA2}
	R(G,V_j,-1)\leq R(G_1,V_j,-1)+R(G_2,N_H(V_i)\setminus V_j,-1), 
	\end{equation}
	by considering the total weight of $G, G_1$ and $G_2$. Then 	 \allowdisplaybreaks
	\begin{align*}
	N(G,-1) &\leq R(G,V_j,-1) 	&(\text{Lemma~\ref{lem1.2}})  \\
	&\leq R(G_1,V_j,-1)+R(G_2,N_H(V_i)\setminus V_j,-1)  	&(\text{Inequality~\ref{eq:gameA2}})  \\ 
	&\leq R(G,N_H(V_i),-1)   	&(\text{Inequality~\ref{eq:gameA1}}). 
	& \qedhere 
	\end{align*} 	 
\end{proof} 

\begin{proof}[Proof of Lemma~\ref{lem2.3}]
	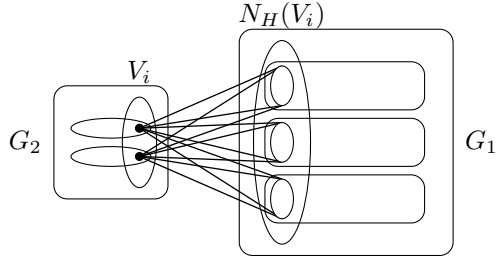
\begin{figure}[H] \centering
		\begin{tikzpicture}[baseline=0.5ex,scale=0.75]	
		\draw[thick,line width=0.5pt] (5.5,1.75) -- (8,3.35);
		\draw[thick,line width=0.5pt] (5.5,1.75) -- (8,2.65);
		\draw[thick,line width=0.5pt] (5.5,1.75) -- (8,2.35);
		\draw[thick,line width=0.5pt] (5.5,1.75) -- (8,1.65);
		\draw[thick,line width=0.5pt] (5.5,1.75) -- (8,1.35);
		\draw[thick,line width=0.5pt] (5.5,1.75) -- (8,0.65);
		\draw[thick,line width=0.5pt] (5.5,2.25) -- (8,3.35);
		\draw[thick,line width=0.5pt] (5.5,2.25) -- (8,2.65);
		\draw[thick,line width=0.5pt] (5.5,2.25) -- (8,2.35);
		\draw[thick,line width=0.5pt] (5.5,2.25) -- (8,1.65);
		\draw[thick,line width=0.5pt] (5.5,2.25) -- (8,1.35);
		\draw[thick,line width=0.5pt] (5.5,2.25) -- (8,0.65);
		\draw[rounded corners=5pt] (7.25,0) rectangle (11,4);		
		\draw[rounded corners=5pt] (4,1) rectangle (6,3);
		\draw[rounded corners=5pt] (7.7,2.575) rectangle (10.5,3.425);
		\draw[rounded corners=5pt] (7.7,1.575) rectangle (10.5,2.425);
		\draw[rounded corners=5pt] (7.7,0.575) rectangle (10.5,1.425);	
		\draw (5.5,2) ellipse (3mm and 8mm);
		\draw (8,2) ellipse (5mm and 18mm);
		\draw[fill=white] (8,3) ellipse (2mm and 3.5mm);
		\draw[fill=white] (8,2) ellipse (2mm and 3.5mm);
		\draw[fill=white] (8,1) ellipse (2mm and 3.5mm);
		\draw (5.5,3.25) node {$V_i$};
		\draw (8,4.25) node {$์N_H(V_i)$};
		\draw (3.5,2) node {$G_2$};
		\draw (11.5,2) node {$G_1$};
		\draw[fill=black] (5.5,1.75) circle (2pt);	
		\draw[fill=black] (5.5,2.25) circle (2pt);
		\draw (5,1.75) ellipse (7mm and 1.75mm);
		\draw (5,2.25) ellipse (7mm and 1.75mm);
		\end{tikzpicture} 
		\caption{The graph $G$ in Lemma~\ref{lem2.3}.} \label{fig:2.3}
	\end{figure}  
	For $i\in[k]$, let $G_1$ be the union of components of $G-V_i$ containing some vertices of $N_H(V_i)$ and let $G_2=G-G_1$. We observe that $N_H(V_i)=V(G_1)\cap N_G(V(G_2))$ and $V_i=V(G_2)\cap N_G(V(G_1))$ are root sets of $G_1$ and $G_2$, respectively, and every vertex in $N_H(V_i)$ is joined to every vertex in $V_i$. Then 
	\allowdisplaybreaks
	\begin{align*}
	N(G,2=-1)&\leq R(G,V_i,-1=2) 	&(\text{Lemma~\ref{lem1.2}})	 \\
	&\leq R(G,N_H(V_i),1)    		&(\text{Lemma~\ref{lemB.2}}).
	& \qedhere 
	\end{align*} 
\end{proof}

\begin{proof}[Proof of Lemma~\ref{lem2.2}] 
	For $i\in[k]$, let $G_1$ be the union of components of $G-N_H(V_i)$ containing some vertices of $V_i$ and let $G_2=G-G_1$. We observe that $V_i=V(G_1)\cap N_G(V(G_2))$ and $N_H(V_i)=V(G_2)\cap N_G(V(G_1))$ are root sets of $G_1$ and $G_2$, respectively, and every vertex in $V_i$ is joined to every vertex in $N_H(V_i)$.	Then 	
	\allowdisplaybreaks
	\begin{align*}
	N(G,2=-1)&\leq R(G,N_H(V_i),-1=2) 	&(\text{Lemma~\ref{lem1.3}})	 \\
	&\leq R(G,V_i,1)   					&(\text{Lemma~\ref{lemB.2}}).
	& \qedhere 
	\end{align*} 
\end{proof}

\begin{proof}[Proof of Lemma~\ref{lem2.1}]
	Let $v\in V(G)$. 
	
	\textbf{Case 1.} There is a cut edge $uv$ incident to $v$. 
	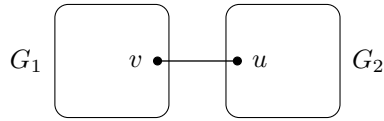
\begin{figure}[H] \centering
		\begin{tikzpicture}[baseline=0.5ex,scale=0.75]	
		\draw[thick,line width=0.5pt] (5.8,2) -- (7.2,2);
		\draw[rounded corners=5pt] (7,1) rectangle (9,3);		
		\draw[rounded corners=5pt] (4,1) rectangle (6,3);
		\draw (3.5,2) node {$G_1$};
		\draw (9.5,2) node {$G_2$};
		\draw[fill=black] (5.8,2) circle (2pt) node[left=2pt] {$v$};	
		\draw[fill=black] (7.2,2) circle (2pt) node[right=2pt] {$u$};
		\end{tikzpicture} 
		\caption{The graph $G$ in Case 1 of Lemma~\ref{lem2.1}.} \label{fig:2.1}
	\end{figure}	
	Let $G_1$ be the component of $G-uv$ containing~$v$ and let $G_2=G-G_1$.  We observe that $\{v\}=V(G_1)\cap N_G(V(G_2))$ and $\{u\}=V(G_2)\cap N_G(V(G_1))$ are root sets of $G_1$ and $G_2$, respectively, and $v$ is adjacent to $u$. Then  \allowdisplaybreaks
	\begin{align*}
	R(G,v,1)	&\geq R(G,u,2=-1)	&(\text{Lemma~\ref{lemB.2}}) \\ 
	&\geq N(G,-1=2) &(\text{Lemma~\ref{lem1.1}}).   
	\end{align*} 		
	
	\textbf{Case 2.} There is no cut edge incident to $v$. 
	
	Then $v\in V_j$ for some $j\in[k]$ and  $N_G(v)=N_H(V_j)$.
	
	\textbf{Case 2.1.}  $|V_j|\geq 2$. 
	
	Therefore, $v$ is a feasible move in the game on $G$. So $G-v$ is connected. Then \allowdisplaybreaks
	\begin{align*}
	R(G,v,1=-2)&= R(G-v,N_G(v)=N_H(V_j),-1) 	&(\text{Observation~\ref{obs3}}) \\
	&\geq N(G-v,-1=1)    	&(\text{Lemma~\ref{lem1.3} by induction}) \\ 
	&\geq N(G,2)     		&(\text{Observation~\ref{obs1}}).
	\end{align*} 

	\textbf{Case 2.2.} $|V_j|=1$. 
	
	Then, by Lemma~\ref{lem2.2},
	\begin{equation*}
		R(G,v,1)=R(G,V_j,1)\geq N(G,2). \qedhere
	\end{equation*}
	
\end{proof}

	We proceed to prove our main theorem.

\begin{proof}[Proof of Theorem~\ref{thm:mainthm}]
	Let $G$ be an even B$(T)$-tree, where $T$ is a tree and let $v\in V(G)$. Then, by Lemmas~\ref{lem1.1} and~\ref{lem2.1}, it follows that
	\begin{equation*}
		N(G,2=-1)\leq R(G,v,-1=2)\leq N(G,1).
	\end{equation*}
	Therefore, Alice wins the game on $G$.
\end{proof}

	We now deduce Corollary~\ref{cor:blowupcycle} from Theorem~\ref{thm:mainthm}.

\begin{proof}[Proof of Corollary~\ref{cor:blowupcycle}] 
	We give a proof by induction on the number of vertices. Let $G$ be an even blow-up of a cycle. We note that every vertex of $G$ is a non-cut vertex. Alice claims a maximum weighted vertex of $G$ in her first move, say a vertex $a$. Let $b$ be the vertex claimed by Bob in his first move. Then $G-\{a,b\}$ is an even blow-up of either a path or a cycle. If $G-\{a,b\}$ is an even blow-up of a path, then Alice wins the game on $G-\{a,b\}$ by Theorem~\ref{thm:mainthm}. Otherwise, Alice wins the game on $G-\{a,b\}$ by the induction hypothesis. In both cases, since $w(a)\geq w(b)$, Alice wins the game on $G$.
\end{proof}

\section{Concluding Remarks} \label{section-conclusion}
 
	We provide two new classes, namely B$(T)$-trees and B$(C_{2n})$, of bipartite even graphs which satisfy Conjecture~\ref{conj}. However, this conjecture is still open. It was shown in~\cite{Egawa:Kmntree} that Lemmas~\ref{lem1.1} and~\ref{lem2.1} are not true for general bipartite graphs, therefore this method cannot be directly used to solve the full conjecture. There are several variants of the graph grabbing game, for example,  the graph sharing game (see~\cite{Chaplick:sharing,Cibulka:sharing,Gagol:sharing,Knauer:pizza,Micek:sharing}), the graph grabbing game on $\{0,1\}$-weighted graphs (see~\cite{Eoh:01weighted}), and the convex grabbing game (see~\cite{Matsumoto:Convex}), where a few problems were left open. 
	
\section*{Acknowledgment}	

	The first author is grateful for financial support from the Science Achievement Scholarship of Thailand. 

	\bibliographystyle{siam} 
	\bibliography{GGG}

\end{document}